\documentclass[11pt,leqno,twoside]{article}

\usepackage{mathrsfs}
\usepackage{amssymb}
\usepackage{amsthm,amsmath,amscd,amsxtra}
\usepackage{color}
\allowdisplaybreaks

\def\rr{{\mathbb R}}
\def\rn{{{\rr}^n}}

\def\cl{{\mathcal L}}

\def\az{\alpha}

\def\gz{{\gamma}}

\def\kz{{\kappa}}

\def\wz{\widetilde}

\def\ls{\lesssim}
\def\gs{\gtrsim}

\def\gfz{\genfrac{}{}{0pt}{}}

\def\r{\right}
\def\lf{\left}
\def\gfz{\genfrac{}{}{0pt}{}}

\newtheorem{thm}{Theorem}[section]
\newtheorem{lem}[thm]{Lemma}
\newtheorem{prop}[thm]{Proposition}
\newtheorem{rem}[thm]{Remark}
\newtheorem{cor}[thm]{Corollary}
\newtheorem{defn}[thm]{Definition}

\def\glip{{\mathop\mathrm{GLip}}}

\numberwithin{equation}{section}

\begin{document}

\arraycolsep=1pt

\title{\bf On the global Gaussian Lipschitz space
\footnotetext{\hspace{-0.22cm}2010
{\it Mathematics Subject Classification}. %
Primary 26A16; Secondary 28C20, 46E35. %
\endgraf {\it Key words and phrases}. %
Gauss measure space, Lipschitz space, Ornstein-Uhlenbeck Poisson kernel.%
\endgraf The first author is supported by
National Natural Science Foundation of China
(Grant No. 11471042).
}}
\author{Liguang Liu and Peter Sj\"ogren}
\date{13 April, 2015}
\maketitle

\begin{center}
\begin{minipage}{13.5cm}\small
 \noindent{\bf Abstract.}
A Lipschitz space is defined in the Ornstein-Uhlenbeck setting, by means of a
bound for the gradient of the Ornstein-Uhlenbeck Poisson integral. This space
is then characterized with a  Lipschitz-type continuity condition. These
functions turn out to have at most logarithmic growth at infinity. The
analogous  Lipschitz space containing only bounded functions was introduced
by Gatto and Urbina and has been characterized by the authors in
 \cite{LS}.
\end{minipage}
\end{center}



\section{Introduction and main result}

Consider the Euclidean space $\rn$ endowed with the Gaussian measure
$\gz$, given by
$$d\gz(x)=\pi^{-n/2}e^{-|x|^2}.$$
The Gaussian analogue of the Euclidean Laplacian is
the \emph{Ornstein-Uhlenbeck operator}
$$\cl =-\frac12 \Delta +x\cdot\nabla ,$$
where $\nabla =  (\partial_{x_1}, \dots, \partial_{x_n}).$
The  heat semigroup generated by $\cl$
 and defined  in  $L^2(\gz)$
 is the so-called \emph{Ornstein-Uhlenbeck semigroup}
$$T_t = e^{-t\cl},\quad t\ge0.$$
The \emph{Ornstein-Uhlenbeck Poisson semigroup} $P_t = e^{-t\sqrt{-\cl}}, \;t\ge0$,  can be defined
 from $\{T_t\}_{t\ge0}$ by subordination  as
\begin{equation*}
  P_tf(x)=\frac1{\sqrt \pi}\int_0^\infty \frac{e^{-u}}{\sqrt u}
  \,T_{t^2/(4u)}f(x)\,du,\qquad x\in\rn,
\end{equation*}
for  $f\in L^2(\gz)$. As explained in Section 2, $P_tf$ is given by integration against a kernel
 $P_t(x,y)$.

Via $\{P_t\}_{t\ge0}$, Gatto and Urbina \cite{GU} introduced the \emph{Gaussian Lipschitz space} $\glip_\az$ for all $\az>0$. We shall always have $\az\in(0,1)$.
Then the definition says that a function $f$ in $\rn$
is  in
$\glip_\az$ if it is
bounded and satisfies
\begin{equation}\label{GLip}
\|t\partial_t P_tf\|_{L^\infty}\le A t^\az,\qquad t>0,
\end{equation}
for some $A>0$.
 These spaces and also Gaussian Besov spaces were studied
 in a series of works; see  \cite{GPU,GU, PU} and also
the authors' paper \cite{LS}.

In \cite{LS}, the authors characterized  $\mathrm{GLip}_\az,\;0<\az<1$,
in terms of a   Lipschitz-type continuity  condition. Indeed, Theorem~1.1 of \cite{LS}
says that $f\in \mathrm{GLip}_\az$ if and only if
there exists a positive constant $K$ such that
\begin{equation}\label{strong-lip}
|f(x)-f(y)|\le K \min\lf\{|x-y|^\az,\;\,
\lf(\frac{|x-y_x|}{1+|x|}\r)^{\frac\az2} +|y_x^\prime|^{\az}\,
\r\},
\qquad x,\:y\in\rn.
\end{equation}
Here and in what follows,  we use a decomposition of $y$ as  $y=y_x+y_x^\prime$,
where $y_x$ is parallel to $x$ and $y_x^\prime $  orthogonal to $x$;
however, if  $x=0$ or $n=1$, we let $y_x=y$ and $y_x^\prime=0$.

As is well known, a condition analogous to  \eqref{GLip} for the standard
Poisson integral characterizes the ordinary Lipschitz space; see
 \cite[Sect. V.4]{S1}. If only bounded functions are considered, one obtains
the inhomogeneous  Lipschitz space, and without the boundedness assumption the
larger, homogeneous  Lipschitz space.

In our setting, we shall see that the condition   \eqref{GLip} without
the boundedness condition defines a Gaussian analogue of the homogeneous
Lipschitz space. Since here no homogeneity is involved, we shall call it
the \emph{global Gaussian Lipschitz space}.

In   \eqref{GLip}, an a priori assumption is needed to assure that  $P_t f$
 exists. Here we apply a recent result
by Garrig\'os,  Harzstein, Signes, Torrea and Viviani
\cite{GHSTV}. Clearly, a measurable function  $f$ in $\rn$ has a
well-defined  Gaussian Poisson integral if
\begin{equation*}
\int P_t(x,y) |f(y)|\,dy<\infty,
\end{equation*}
for all $x\in \rn$ and $t>0$. Theorem 1.1
of \cite{GHSTV} says that this  is equivalent
to the growth condition
\begin{equation}\label{cond:GHSTV}
\int_{\rn} \frac{e^{-|y|^2}}{\sqrt{\ln(e+|y|)}}\,|f(y)|\,dy<\infty.
\end{equation}
Moreover,  \eqref{cond:GHSTV} ensures that  $P_t f(x) \to f(x)$ as $t\to 0$
for a.a. $x\in \rn$.

We can now define the global Gaussian Lipschitz space.

\begin{defn}
Let $\az\in(0,1)$. A measurable function  $f$ defined in $\rn$ and
 satisfying \eqref{cond:GHSTV} belongs to
the global Gaussian Lipschitz space $\mathrm{GGLip}_\az$ if
\eqref{GLip} holds.
The corresponding norm is
 $$
\|f\|_{\mathrm{GGLip}_\az}
=\inf\{A>0:\, A \,\;\textup{satisfies}\,\;   \eqref{GLip}\}.
$$
\end{defn}

Strictly speaking, this space consists of functions modulo constants.
A natural question is now what continuity condition characterizes this space.
To state the answer, we start in one dimension and introduce a distance by
\begin{equation}\label{d}
d(x,y) = \left|  \int_{x}^y \frac{d \xi} {1+|\xi|}\right|, \qquad x, y \in \rr.
\end{equation}
Then
\begin{equation*}
d(x,y) = \left| \ln(1+ |x|) - \mathrm{sgn}\, xy \;  \ln(1+ |y|)  \right|
\end{equation*}
for all $x, y \in \rr$, provided we define $\mathrm{sgn}\, 0= 1$.
In several dimensions, we use this distance on the  line spanned by $x$,
defining
\begin{equation*}
d(x,y_x) = \left| \ln(1+ |x|) - \mathrm{sgn} \langle x,y\rangle   \ln(1+ |y_x|)  \right|, \qquad x, y \in \rn,
\end{equation*}
with $y_x$ as before.

Our result reads as follows.

\begin{thm}\label{thm1.2}
Let $\az\in(0,1)$ and let  $f$ be a measurable function  in $\rn$. The following are
equivalent:
\begin{enumerate}
\item[\rm (i)] $f$ satisfies \eqref{cond:GHSTV} and $f\in \rm{GGLip}_\az$;
\item[\rm (ii)]  There exists a positive constant $K$ such that
\begin{equation}\label{eq:gglip}
|f(x)-f(y)|
\le K\min\lf\{
|x-y|^\alpha,\; d(x, y_x)^{\frac{\az} 2}
+|y_x'|^\alpha
\r\},
\qquad x,\:y\in\rn,
\end{equation}
 after correction of $f$ on a null set.
\end{enumerate}
Moreover,
\begin{equation}\label{norm}
\|f\|_{\mathrm{GGLip}_\az} \simeq
\inf\{K:\, K\,\;\textup{satisfies}\,\; \eqref{eq:gglip}\}.
\end{equation}
\end{thm}
The meaning of the symbol $\simeq$ is explained below.

\begin{rem}\rm
To compare  \eqref{eq:gglip} and  \eqref{strong-lip}, one
easily verifies that if $\langle x,y\rangle > 0$ and $1/2 < |x|/|y_x| < 2$,
then
\begin{equation}\label{comp}
d(x,y_x) \simeq \frac{|x-y_x|}{1+|x|}.
\end{equation}
Moreover,
the space $\rm{GLip}_\alpha$  can be described in terms of the distance function $d$.
Indeed, as \eqref{strong-lip} implies boundedness (see \cite[Lemma~2.1]{LS}),
it is easy to check that \eqref{strong-lip} holds if and only if
there exists a constant $K'>0$ such that
$$|f(x)-f(y)|
\le K'\min\lf\{1,\;
|x-y|^\alpha,\; d(x, y_x)^{\frac{\az} 2}
+|y_x'|^\alpha
\r\}$$
for all $x,\,y\in\rn$.
This also tells us that for bounded functions,
  \eqref{strong-lip} is equivalent to \eqref{eq:gglip}.
But \eqref{eq:gglip} implies only that
\[
f(x) = O((\ln |x|)^{\alpha/2}) \qquad \mathrm{as} \qquad |x| \to \infty.
\]
This condition is sharp, as shown by an example in Section 5; observe that
it is much stronger than  \eqref{cond:GHSTV}.
\end{rem}

The paper is organized as follows. Section~2 contains a needed improvement
of the estimate for $P_t(x,y)$ and its derivatives in  \cite{LS}. Some
properties of the Gaussian Poisson integral are obtained in Section~3. Then
Theorem~\ref{thm1.2} is proved in Section~4. Finally, we give in Section~5
an example of a function in $\rm{GGLip}_\az$ with logarithmic growth.

\medskip

\noindent \textbf{Notation.\,}
Throughout the paper, we shall write $C$ for various positive constants which
depend only on $n$
and $\az$, unless otherwise explicitly stated.
Given any two nonnegative quantities $A$ and $B$, the notation
$A\lesssim B$ stands for $A\le C B$ (we say that $A$ is controlled by $B$),
and $A\gtrsim B$  means $B\ls A$.
If $B\ls A\ls B$, we write $A\simeq  B$.

For positive quantities $X$,  we shall write
$\exp^*(-X),$
meaning  $\exp(-cX)$ for some constant
$c=c(n,\az)>0$.

\section{The Ornstein-Uhlenbeck Poisson kernel}\label{sect2}

It is known that  for  $f\in L^2(\gz)$,
$$T_tf(x) =\frac1{\pi^{n/2}}\int_\rn M_{e^{-t}}(x,y)f(y)\,dy, \qquad x\in\rn,\;\;t>0,$$
where $M_{e^{-t}}$ is the \emph{Mehler kernel} defined  by
$$M_r(x,y)=\frac{e^{-\frac{|y-rx|^2}{1-r^2}}}{(1-r^2)^{n/2}},
\qquad x,\,y\in\rn,\quad 0<r<1.$$
The Gaussian Poisson integral $ P_tf$ is given by an integral kernel called
 the \emph{Ornstein-Uhlenbeck Poisson kernel} and denoted  by $P_t(x,y)$; thus
\begin{equation*}
  P_tf(x)=\int_{\rn} P_t(x,y) f(y)\, dy, \qquad x\in\rn,\;\;t>0.
\end{equation*}
Because of the subordination formula,  $P_t(x,y)$ is given by
\begin{align}
P_t(x,y) &=
\frac1{ \pi^{(n+1)/2}}\int_0^\infty \frac{e^{-u}}{\sqrt u}
  \,M_{e^{-t^2/(4u)}}(x,y)\,du \notag \\
&= \frac1{2\pi^{(n+1)/2}} \int_0^\infty
\frac{t}{s^{3/2}} \,e^{-\frac{t^2}{4s}}\,
\frac{\exp(-\frac{|y-e^{-s}x|^2}{1-e^{-2s}})}{(1-e^{-2s})^{n/2}} \,ds.
\label{poisson-ker1}
\end{align}
Here we  inserted the expression for the Mehler kernel
and transformed the variable.

The following estimate for $P_t$ and its first derivatives  is established in \cite[Theorems~1.2 and 1.3]{LS}.

\begin{prop}\label{prop2.1}
For all  $ t>0$, \hskip4pt $x,\:y\in \rn$
 and $i\in\{1,2,\dots,n\}$,
the kernel $P_t$ satisfies
\begin{eqnarray*}
P_t(x, y) +|t\partial_t P_t(x,y)|+|t\partial_{x_i} P_t(x,y)|
\le C \lf[K_1(t,x,y)+K_2(t,x,y)+K_3(t,x,y)+K_4(t,x,y)\r],
\end{eqnarray*}
where
 \begin{eqnarray*}
K_1(t,x,y) &=& \frac t{(t^2+|x-y|^2)^{(n+1)/2}}\,\exp^\ast\lf(-t(1+|x|)\r); \\
K_2(t,x,y) &=&\frac t{|x|} \left(t^2+ \frac{|x-y_x|}{|x|}+|y_x^\prime |^2\right)^{
-\frac{n+2}2}\,\exp^\ast\lf(-\frac{(t^2+|y_x'|^2)|x|}{|x-y_x|}\r)\,
\chi_{\{|x|>1,\; x\cdot y>0,\, |x|/2\le |y_x|<|x|\}};\\
 K_3(t,x,y) &=& \mathrm{min}(1,t)\, \exp^\ast(-|y|^2);\\
K_4(t,x,y) &=&\frac t{|y_x|} \left(\ln \frac {|x|} {|y_x|}\right)^{ -\frac32}  \,
\exp^\ast\left(-\frac {t^2}{\ln \frac  {|x|} {|y_x|}}\right) \,
 \exp^\ast(-|y_x^\prime |^2) \, \chi_{\{x \cdot y>0,\:1<|y_x|<|x|/2\}}.
 \end{eqnarray*}
\end{prop}

We need a slight sharpening of this lemma. The term $K_3$
will be modified
to decay for large $x$.

\begin{lem}\label{lem2.2}
The estimate of Proposition \ref{prop2.1} remains valid if the kernel
$K_3(t,x,y)$  is replaced by
\begin{eqnarray*}
\wz{K_3}(t,x,y)= \min\lf\{1,\,\frac{t }{ [\ln(e+|x|)]^{1/2}}\r\}\exp^\ast(-|y|^2).
\end{eqnarray*}
\end{lem}

\begin{proof}
From the proof of \cite[Theorem~1.3]{LS}, we see that  $|t\partial_t P_t(x,y)|$ and $|t\partial_{x_i} P_t(x,y)|$ can be controlled by an integral similar to the right-hand side of \eqref{poisson-ker1}
(only with $\exp$ in \eqref{poisson-ker1} replaced by $\exp^\ast$). Thus, we only need to consider $P_t(x,y)$.

When $|x|\le 4+2|y_1|$,
we have
$
\exp^\ast(-|y|^2) \ls \exp^\ast(-|y|^2) \exp^\ast(-|x|^2)
$
and hence
$K_3(t,x,y)\ls \wz{K_3}(t,x,y).$

Thus we assume from now on that  $|x| > 4+2|y_1|$. We shall sharpen a
few arguments in the proof of  \cite[Proposition~4.1]{LS}.
By the rotation invariance of $P_t(x,y)$ and $\wz{K_3}(t,x,y)$, we may assume
that $x=(x_1,0,\dots,0)$ with $x_1 > 0$.
 The decomposition
of $y$ will then be written $y = (y_1,0,\dots,0) + (0,y')$, and $|y_1| < x_1/2$.

\medskip

\noindent {\it Case 1.}\, $-x_1/2<y_1\le 0$.\,
Using the notation from  the proof of \cite[Proposition~4.1(i)]{LS},
we see that we only need to verify that $J_2 \ls \wz{K_3}$.
By \cite[formula (4.9)]{LS} and the fact that $y_1 \le 0< x_1$, we have
$ $
\begin{eqnarray}
J_2 &&\simeq \exp^\ast(-|y'|^2)\,\int_{\ln  2}^\infty
\frac {t}{s^{3/2}}\,\exp^\ast\lf(-\frac{t^2}{s}\r)\,
\exp^\ast(-|y_1-e^{-s}x_1|^2) \,ds \notag\\
&&\ls \exp^\ast(-|y|^2)\,\int_{\ln  2}^\infty
\frac {t}{s^{3/2}}\,\exp^\ast\lf(-\frac{t^2}{s}\r)\,
\exp^\ast(-e^{-2s}x_1^2) \,ds. \label{j2}
\end{eqnarray}
Note that
\begin{eqnarray*}  
\int_{\frac12\ln  x_1}^\infty
\frac {t}{s^{3/2}}\,\exp^\ast\lf(-\frac{t^2}{s}\r)\,
\exp^\ast(-e^{-2s}x_1^2) \,ds
&&\simeq \int_{\frac12\ln  x_1}^\infty
\frac {t}{s^{3/2}}\,\exp^\ast\lf(-\frac{t^2}{s}\r)\,
 \,ds\\
 &&\ls \min\lf\{1,\, t  (\ln x_1)^{-1/2}\r\} \notag
\end{eqnarray*}
and
\begin{eqnarray*}   
\quad\int_{\ln  2}^{\frac12\ln  x_1}
\frac {t}{s^{3/2}}\,\exp^\ast\lf(-\frac{t^2}{s}\r)\,
\exp^\ast(-e^{-2s}x_1^2) \,ds
&&\le \exp^\ast(-x_1)  \int_{\ln  2}^{\frac12\ln  x_1}
\frac {t}{s^{3/2}}\,\exp^\ast\lf(-\frac{t^2}{s}\r)\,
 \,ds\\
&&\ls\exp^\ast(-x_1) \min\{1,\,t\}, \notag
\end{eqnarray*}
from which the required estimate follows.

\medskip

\noindent {\it Case 2:}\, $0<y_1<x_1/2$.\,  Considering now the proof of
\cite[Proposition~4.1(iii)]{LS},
we only need to estimate the terms $J_{2,1}^{(2)}$ and $J_{2,3}$,
 and also  $J_{2,2}$ when $y_1\in(0,1]$.

From  \cite[formula (4.16)]{LS}, we get for  $y_1\in(0,1]$,
 \begin{eqnarray*}
 J_{2,2}
&&\simeq  \frac t{(\ln \frac{x_1}{y_1})^ {3/2}}
\exp^*{ \left(-\frac{t^2}{\ln \frac{x_1}{y_1}}\right)} \exp^*{(-|y' |^2)}\\
&&\ls
 \min\lf\{\frac t{(\ln \frac{x_1}{y_1})^{3/2}} ,\, \frac 1{\ln \frac{x_1}{y_1}} \r\}
 \exp^*{(-|y|^2)}\\
&&\ls \wz{K_3}(t,x,y),
\end{eqnarray*}
since here $\ln \,(x_1/y_1) \gs \ln\,(e+|x|)$. Further,
\begin{eqnarray}\label{eq2.4}
J_{2,1}^{(2)}+J_{2,3}
&& \le \exp^\ast(-|y'|^2)\,
\int   \frac {t}{s^{3/2}}\,\exp^\ast\lf(-\frac{t^2}{s}\r)\,
\exp^\ast(-|y_1-e^{-s}x_1|^2) \,ds,
\end{eqnarray}
where the integral is taken over the set $\{s> \ln 2: |s-\ln\, (x_1/y_1)| > c_0 \}$,
for some $c_0>0$. Thus the quotient $e^{-s}x_1/y_1$ stays away from $1$ in this
integral, so that
$|y_1-e^{-s}x_1| \simeq \max\{e^{-s} x_1,\, y_1 \}\simeq e^{-s}x_1+y_1$.
This implies that the right-hand side of \eqref{eq2.4} is controlled by the expression in \eqref{j2} and thus by  $\wz{K_3}$.

Lemma \ref{lem2.2} is proved.
\end{proof}


\section{Auxiliary lemmas}\label{sect3}

\begin{lem}\label{lem3.1}
There exists a constant $C>0$ such that for all $x,\,y\in\rn$ and $t>0$,
\begin{equation*}
|\partial_tP_t(x,y)|
\le C\, \frac1{t}\,P_{t/2}(x,y).
\end{equation*}
\end{lem}

\begin{proof}
Differentiating \eqref{poisson-ker1}, we get
\begin{equation*}     
\partial_{t} P_t(x,y)
= \frac1{2\pi^{(n+1)/2}}\,\frac1{t} \int_0^\infty \frac {t}{s^{3/2}}\,e^{-\frac{t^2}{4s}}\,
\lf(1-\frac{t^2}{2s}\r)
\frac{e^{-\frac{|y-e^{-s}x|^2}{1-e^{-2s}}}}{(1-e^{-2s})^{n/2}} \,ds.
\end{equation*}
It is now enough to observe that
\[
e^{-\frac{t^2}{4s}}\,
\lf|1-\frac{t^2}{2s} \r| \ls e^{-\frac{{(t/2)}^2}{4s}}
\]
and compare with \eqref{poisson-ker1}.
\end{proof}

\begin{lem}\label{lem3.2}
Fix $i\in\{1,2,\dots,n\}$ and let $R>0$. Then there exists a constant
$C>0$, depending only on $n$ and $R$, such that
 for all $x,\:y\in\rn$ with $|x|< R$,
\begin{equation}\label{a}
|\partial_{x_i}P_t(x,y)|
\le C\, (1 + t^{-4-n}) P_{t/2}(x,y), \qquad t>0,
\end{equation}
and
\begin{equation}\label{b}
|\partial_{x_i}P_t(x,y)|
\le C  t^{-1/2}e^{-|y|^2} [\ln (e+|y|)]^{-3/4}, \qquad t>1.
\end{equation}

\end{lem}

\begin{proof}
In this proof, all constants denoted $C$ will depend only on  $n$ and $R$,
and the same applies to the implicit constants in the $\lesssim$ and
$\simeq$ symbols.
We let  $|x|< R$, and we can clearly assume that $R>1$.

Differentiating \eqref{poisson-ker1}, we get
\begin{equation}\label{xj}
\partial_{x_i}P_t(x,y)= \frac1{\pi^{(n+1)/2}} \int_0^\infty
\frac{t}{s^{3/2}} \,e^{-\frac{t^2}{4s}}\, \frac{e^{-s}(y_i-e^{-s}x_i)}{1-e^{-2s}}\,
\frac{\exp(-\frac{|y-e^{-s}x|^2}{1-e^{-2s}})}{(1-e^{-2s})^{n/2}} \,ds.
\end{equation}
Compared with  \eqref{poisson-ker1}, the integral has now an extra factor
${e^{-s}(y_i-e^{-s}x_i)}/{(1-e^{-2s})}$.

With $\gamma>0$, we shall use repeatedly the simple inequality
\begin{equation}
  \label{simple}
e^{-\frac{t^2}{4s}} \le C_\gamma \left(\frac{s}{t^2}\right)^\gamma e^{-\frac{(t/2)^2}{4s}}
\end{equation}
for some $ C_\gamma>0$, and here we sometimes drop the last factor.

We start with the simple case of bounded $y$; more precisely we assume
$|y|\le e^{12}\,R$. Then the extra
factor is no larger than $Ce^{-s}/{(1-e^{-2s})}$. An application of  \eqref{simple}
with  $\gamma = 1+n/2$ yields
\[
|\partial_{x_i}P_t(x,y)|  \lesssim t^{-2-n} \int_0^\infty \frac{t}{s^{3/2}} \,
e^{-\frac{(t/2)^2}{4s}}\frac{e^{-s}s^{1+n/2}}{(1-e^{-2s})^{1+n/2}}\,
\exp\left(-\frac{|y-e^{-s}x|^2}{1-e^{-2s}}\right)\,ds.
\]
Comparing with  \eqref{poisson-ker1}, one  sees that
this estimate implies  \eqref{a}. If we choose instead  $\gamma = 2+n/2$,
\eqref{b} will also follow, since $y$ stays bounded.

From now on, we assume that  $|y|>   e^{12}\,R$. Then  \eqref{xj} implies
\begin{equation}\label{xjj}
|\partial_{x_i}P_t(x,y)| \lesssim \int_0^\infty
\frac{t}{s^{3/2}} \,e^{-\frac{t^2}{4s}}\, \frac{e^{-s}|y|}{1-e^{-2s}}\,
\frac{\exp(-\frac{|y-e^{-s}x|^2}{1-e^{-2s}})}{(1-e^{-2s})^{n/2}} \,ds.
\end{equation}

We first estimate the exponent
\begin{equation*}
 E(s,x,y) =  -\frac{|y-e^{-s}x|^2}{1-e^{-2s}}
\end{equation*}
from
  \eqref{xjj}.   It satisfies
\begin{equation*}
 E(s,x,y)  \le  \frac{-|y|^2 + 2e^{-s} y\cdot x}{1-e^{-2s}}
\le \frac{-|y|^2 + \frac12  e^{-2s}|y|^2 + 2 |x|^2}{1-e^{-2s}},
 \end{equation*}
where we applied the inequality between the geometric and arithmetic means.
 If  $e^{-s}<1/2$, then
\[
E(s,x,y)  \le  \frac{-|y|^2 + \frac12  e^{-2s}|y|^2}{1-e^{-2s}} + C.
\]
If instead $e^{-s}\ge 1/2$, we have $2 |x|^2 < e^{-2s}|y|^2/4 $
since $|y|>e^{12}|x|$,  and thus
\[
E(s,x,y)  \le  \frac{-|y|^2 + \frac34  e^{-2s}|y|^2}{1-e^{-2s}}.
\]
In both cases,
\[
E(s,x,y)  \le  -|y|^2 \: \frac{1 - \frac34  e^{-2s}}{1-e^{-2s}} + C
\le  -|y|^2\left(1 + \frac14  e^{-2s}\right) + C,
\]
and this implies
\begin{equation}
  \label{expon}
 e^{E(s,x,y)} \lesssim e^{-|y|^2} \min\left(1,  \frac{e^{2s}}{|y|^2} \right).
\end{equation}

We also need a converse inequality, under the assumption that $s>\ln|y|$.
Then
\begin{equation}  \label{converse}
 E(s,x,y)  \ge  \frac{-|y|^2 - 2e^{-s} |y||x| -e^{-2s}|x|^2}{1-e^{-2s}}
\ge \frac{-|y|^2}{1-|y|^{-2}} - C \ge-|y|^2 - C.
  \end{equation}

Now split the integral in  \eqref{xjj} as
\[
\left(\int_0^3 + \int_3^{\ln{|y|}} +\int_{\ln{|y|}}^\infty\right) \frac{t}{s^{3/2}} \,e^{-\frac{t^2}{4s}}\, \frac{e^{-s}|y|}{1-e^{-2s}}\,
\frac{\exp(-\frac{|y-e^{-s}x|^2}{1-e^{-2s}})}{(1-e^{-2s})^{n/2}} \,ds
= I_1 +  I_2 +  I_3,
\]
say; observe that  $\ln|y|>12$.
We shall prove that these three integrals  satisfy the bounds in
  \eqref{a} and \eqref{b}.

In $I_3$, we have
${e^{-s}|y|}/{(1-e^{-2s})} \lesssim 1$.
Comparing with  \eqref{poisson-ker1}, we conclude that
\[
I_3 \lesssim  P_{t}(x,y) \lesssim  P_{t/2}(x,y),
\]
which is part of  \eqref{a}. Aiming at  \eqref{b}, we apply  \eqref{simple}
with   $\gamma = 3/4$ and
 \eqref{expon}, where the  minimum is 1, to conclude that
\[
 I_3 \lesssim \int_{\ln{|y|}}^\infty t^{-1/2}\, s^{-3/4}\, e^{-s}\,|y|\, e^{-|y|^2} \,ds
\lesssim  t^{-1/2} \, (\ln |y|)^{-3/4}\, e^{-|y|^2},
\]
as desired.

To deal with $ I_2$, we apply  \eqref{expon},
now with the second quantity in the minimum, and obtain
\begin{equation}
  \label{middle}
I_2 \lesssim \int_3^{\ln|y|} \frac{t}{s^{3/2}} \,
e^{-\frac{t^2}{4s}}  \,   
\frac{e^s}{|y|}\,e^{-|y|^2} \,ds.
   \end{equation}
Using  \eqref{simple}, again with   $\gamma = 3/4$, we can
 estimate this integral  by
\[
 t^{-1/2} e^{-|y|^2} \int_3^{\ln|y|} s^{-3/4}\, \frac{e^s}{|y|} \,ds,
\]
which gives the bound in  \eqref{b} for  $ I_2$.
Thinking of  \eqref{a}, we write
the integral in  \eqref{middle} as
\[
 t e^{-|y|^2}|y|^{-1} \int_3^{\ln|y|}  \phi(s)
e^{s/2} \,ds,
\]
where
\[
 \phi(s) = \frac{ e^{s/2}}{s^{3/2}} \,
e^{-\frac{t^2}{4s}}.
\]
Here both the factors are increasing functions of $s$ in $(3,\infty)$,
and so is $\phi$. Thus for any $\eta \in (0,1)$,
\[
\sup_{(3,\ln|y|)} \phi(s) \le \phi(\eta + \ln|y|),
\]
and so
\[
 I_2 \lesssim t e^{-|y|^2}|y|^{-1}  \phi(\eta +\ln|y|)  \int_3^{\ln|y|}  e^{s/2} \,ds
\simeq  t e^{-|y|^2} \frac{ 1}{(\eta +\ln|y|)^{3/2}} \,
e^{-\frac{t^2}{4(\eta +\ln|y|)}}.
\]
Integrating in $\eta,$ we see that
\begin{equation}
  \label{dd}
 I_2   \lesssim \int_{\ln|y|}^{1+\ln|y|} \frac{t}{s^{3/2}}\, e^{-\frac{t^2}{4s}} \,
e^{-|y|^2}\,ds.
\end{equation}
Because of \eqref{converse}, this integral is dominated by the one defining
  $P_t(x,y)$  in  \eqref{poisson-ker1}. Since
$P_t(x,y) \lesssim P_{t/2}(x,y)$,
it follows that  $ I_2   \lesssim P_{t/2}(x,y)$.

Finally, we estimate  $I_1 $ by means of  \eqref{expon}.
Since here $1-e^{-2s} \simeq s$, we get
\[
I_1 \lesssim \int_0^{3} \frac{t}{s^{3/2}} \, e^{-\frac{t^2}{4s}} \,\frac 1{|y| \,s^{1+n/2}}\,
e^{-|y|^2}\,ds.
\]
Using  \eqref{simple} with   $\gamma = 2+n/2$,
we conclude that
\begin{equation}
  \label{03}
I_1   \lesssim  t^{-3-n} \int_0^{3}  s^{-1/2} e^{-\frac{(t/2)^2}{4s}} \frac 1{|y|}\,e^{-|y|^2} \,ds.
\end{equation}
This leads immediately to the bound in  \eqref{b}. For  \eqref{a},
we can estimate the right-hand side in \eqref{03} by
\[
  t^{-3-n} \frac 1{|y|}\, e^{-\frac{t^2}{48}}\, e^{-|y|^2}
 \lesssim t^{-4-n}\, \frac{t}{(\eta + \ln|y|)^{3/2}} \, e^{-\frac{t^2}{4(\eta + \ln|y|)}} \,
e^{-|y|^2}
\]
with  $\eta \in (0,1)$ as before, since $\ln|y|>12$. As a result, we get a bound
for  $I_1 $
 similar to  \eqref{dd}
but with an extra factor $ t^{-4-n}$, and thus also  the bound in \eqref{a}.

Lemma \ref{lem3.2} is proved.
\end{proof}

\begin{prop}\label{prop3.3}
Let $f$ be a measurable function on $\rn$ satisfying \eqref{cond:GHSTV}.
Then for all $i\in\{1,2,\dots,n\}$ and $x\in\rn$ ,
\begin{equation}\label{eq3.15}
\partial_{x_i}\partial_tP_{s+t}f(x)
=\int_{\rn} \partial_{x_i}P_s(x,y)\,\partial_tP_tf(y)\,dy, \qquad s,\,t > 0,
\end{equation}
and
\begin{equation}\label{eq3.16}
\lim_{t\to\infty}\partial_{x_i}P_tf(x)=0.
\end{equation}
\end{prop}

\begin{proof}
We can assume $|x|< R$ for some $R>0$ and thus apply the estimates from
Lemma  \ref{lem3.2}.
First we verify the absolute convergence of the integral in \eqref{eq3.15},
by showing that
\begin{eqnarray*}
\int_{\rn}\int_{\rn} |\partial_{x_i}P_s(x,y)| |\partial_{t}P_t(y,z)||f(z)|\,dy\,dz
<\infty.
\end{eqnarray*}
 Lemmas \ref{lem3.2} and  \ref{lem3.1} imply that
this integral  is, up to a factor $C(n,R)$, no larger than
\begin{eqnarray*}
\frac{1+s^{-4-n}}{t}\int_{\rn}\int_{\rn} P_{s/2}(x,y) P_{t/2}(y,z)|f(z)|\,dy\,dz
= \frac{1+s^{-4-n}}{t}\int_{\rn} P_{(s+t)/2}(x,z)|f(z)|\,dz < \infty,
\end{eqnarray*}
where
 the equality comes from the semigroup property.  The last integral  here
is finite because of
 \eqref{cond:GHSTV}; indeed,  \cite[formula (6.4)]{GHSTV} says that
$P_t(x,y)$ is controlled by  $
e^{-|y|^2}/\sqrt{\ln(e+|y|)
}
$,  locally uniformly in $x$ and $t$.

Our next step consists in integrating the right-hand side of  \eqref{eq3.15}
along intervals in the variables $x_i$ and $t$. We choose two points
 $x',\,x''\in\rn$ with  $|x'|,\,|x''| < R$   which differ only in the
$i$:th coordinate, and also two points
$t',\,t''>0$. Fubini's theorem applies because of the above estimates, and
we get
\begin{eqnarray*}
&&\int_{x_i'}^{x_i''}\int_{t'}^{t''} \lf(\int_{\rn}\int_{\rn} \partial_{x_i}P_s(x,y)\partial_{t}p_t(y,z)f(z)\,dy\,dz\r)\,dt\,dx_i\\
&&\quad =  \int_{\rn} \int_{\rn} [P_s(x'',y)-P_s(x',y)]\,[P_{t''}(y,z)-P_{t'}(y,z)]f(z)\,dy\,dz\\
&&\quad = P_{s+t''}f(x'')-P_{s+t''}f(x')- P_{s+t'}f(x'')+P_{s+t'}f(x').
\end{eqnarray*}
From this, we obtain   \eqref{eq3.15} by differentiating with respect to
$x_i''$ and $t''$.

Finally,  \eqref{eq3.16} is a direct consequence of  \eqref{b}
and  \eqref{cond:GHSTV}.
\end{proof}

Proposition \ref{prop3.3} now allows us to apply the method of proof of
 \cite[Proposition~3.2]{LS} and obtain the same estimates as there.

\begin{cor}\label{cor:3.4}
Let $\az\in(0,1)$ and let $f\in\mathrm{GGLip}_\az$ with norm 1.
\begin{enumerate}
\item[\rm (i)] For all $i\in\{1,2,\dots,n\}$, $t>0$ and $x\in\rn$,
\begin{eqnarray*}
|\partial_{x_i} P_tf(x)|
\le C t^{\az-1}.
\end{eqnarray*}
\item[\rm (ii)] For all $t>0$ and $x=(x_1,0,\dots,0)\in\rn$ with $x_1\ge 0$,
\begin{eqnarray*}
|\partial_{x_1} P_tf(x)|
\le C t^{\az-2}(1+x_1)^{-1}.
\end{eqnarray*}
\end{enumerate}
\end{cor}

\section{Proof of Theorem \ref{thm1.2}}\label{sect4}

$\rm (i) \Longrightarrow \rm (ii)$:    
We assume that $f$ satisfies  \eqref{cond:GHSTV}
and \eqref{GLip}.
   According to \cite[Theorem~1.1]{GHSTV},  $P_tf(x)\to f(x)$ as $t\to 0$
for a.a. $x\in\rn$, and thus we can modify $f$ on a null set so that
 this convergence holds for all $x$.

Now fix  $x,\,y\in\rn$. For all $t>0$, we write
\begin{eqnarray}\label{eq4.1}
|f(x)-f(y)|\le  |f(x)-P_{t}f(x)|+\lf| P_{t}f(x)-P_{t}f(y)\r|+
\lf|P_{t}f(y)-f(y)\r|.
 \end{eqnarray}
 Using Corollary \ref{cor:3.4} (i) and
 arguing as in the verification of \cite[formula (3.7)]{LS}, we
 get
 \begin{eqnarray}\label{eq4.2}
|f(x)-f(y)|\ls |x-y|^\az.
 \end{eqnarray}
To obtain  \eqref{eq:gglip}, it is then enough to
prove that
$$
|f(x)-f(y)| \ls d(x,y_x)^{\frac{\az} 2} + |y_x'|^\alpha.
$$
By writing
$$
|f(x)-f(y)|\le |f(x)-f(y_x)|+ |f(y_x)-f(y)|
$$
and applying \eqref{eq4.2} to the last term here, we see that we need only
 verify that
$$
|f(x)-f(y_x)| \ls  d(x,y_x)^{\frac{\az} 2}.
$$
Making a  rotation, we can assume that  $x=(x_1,0,\dots,0)$ with $x_1\ge0$
and  $y_x=(y_1,0,\dots,0)$.

We estimate $|f(x)-f(y_x)|$ as in  \eqref{eq4.1}. Of the three terms we then get,  the first and
third are controlled by $t^\alpha$. To the second term, we apply Corollary
\ref{cor:3.4} (ii) and the one-dimensional integral expression  \eqref{d}
for $d$. As a result,
$$
|f(x)-f(y_x)| \ls t^\alpha + t^{\alpha-2} d(x,y_x),
$$
and here we choose $t =d(x,y_x)^{1/2} $. This leads to  \eqref{eq:gglip},
and the implication  $\rm (i) \Longrightarrow \rm (ii)$    is proved.

\medskip

\noindent $\rm (ii) \Longrightarrow \rm (i)$:
Letting $y=0$, we see that \eqref{eq:gglip}
implies that
$f(x) = O((\ln|x|)^{\alpha/2})$ as $|x| \to \infty$  and thus also
 \eqref{cond:GHSTV}.
We must verify
\eqref{GLip}.

Using  the fact that  $\int_\rn \partial_tP_t(x,y)\,dy=0$ and
Lemma \ref{lem2.2}, we can write
\begin{eqnarray*}    
|t\partial_tP_tf(x)|
&&= \lf|\int_\rn t \partial_t P_t(x,y)[f(y)-f(x)]\,dy\r|\\
&&\ls  \int_\rn [K_1(t,x,y)+K_2(t,x,y)+ \wz{K_3}(t,x,y)+K_4(t,x,y)] |f(y)-f(x)|\,dy \notag.
\end{eqnarray*}

We thus get four integrals to control by  $t^\alpha$.
For $\int_{\rn} K_1(t,x,y) |f(y)-f(x)|\,dy $, we can apply the same simple argument
as in  \cite[end of Section 3]{LS}, since it uses only the quantity
$|x-y|^\az$ in \eqref{eq:gglip}.

The integral involving $K_2(t,x,y)$ can
also be estimated as in   \cite{LS}, because  \eqref{comp} applies in
the support of  $K_2(t,x,y)$.

For the integral with $\wz{K_3}(t,x,y)$, we apply the inequality
$(a+b)^\kz \le a^\kz + b^\kz$ with $a,b >0$ and $\kz = \az/2 \in (0,1)$ to the expression in
  \eqref{eq:gglip} and get
\begin{eqnarray*}
&&\int_\rn \wz{K_3}(t,x,y) |f(y)-f(x)|\,dy\\
&& \quad \ls  \min\lf\{1,\,\frac{t}{\sqrt{\ln(e+|x|)}}\r\}\int_{\rn}
\lf(\lf(\ln(1+|x|)\r)^{\frac{\az} 2}+\lf(\ln(1+|y_x|)\r)^{\frac{\az} 2}
+|y_x'|^\alpha\r)
\exp^\ast(-|y|^2)\,dy\notag
\end{eqnarray*}
The minimum here is no larger than $t^\az/[\ln(e+|x|)]^{\frac{\az} 2}$,
which leads immediately to the bound $t^\az$ for the whole expression.

Finally,
\begin{eqnarray}\label{eq4.3}
\qquad\int_\rn K_4(t,x,y)|f(y)-f(x)|\,dy
&& \le  \int_{\gfz{x \cdot y>0}{1<|y_x|<|x|/2}}
\frac t{|y_x|} \left(\ln \frac {|x|} {|y_x|}\right)^{ -\frac32}  \,
\exp^\ast\left(-\frac {t^2}{\ln \frac  {|x|} {|y_x|}}\right) \,\\
&&\quad\times
 \exp^\ast(-|y_x^\prime |^2)
\lf(\lf[\ln(1+|x|)-\ln(1+|y_x|)\r]^{\frac{\az} 2}
+|y_x'|^\alpha\r)dy\notag
\end{eqnarray}
When $1<|y_x|<|x|/2$, we have
$$
|\ln(1+|x|)-\ln(1+|y_x|)|=\ln\frac{1+|x|}{1+|y_x|} \simeq \ln\frac{|x|}{|y_x|}.
$$
 After a rotation, we can assume that $x = (x_1,0,\dots,0)$ with $x_1>0$,
so that  $y_x = (y_1,0,\dots,0)$ and $y'_x = (0,y') $ and we have
  $1<y_1<x_1/2$.
 The right-hand integral in
\eqref{eq4.3} is bounded by a constant times
$$
\int_1^{x_1/2} \int_{\rr^{n-1}}
\frac t{y_1} \left(\ln \frac {x_1} {y_1}\right)^{ -\frac32}  \,
\exp^\ast\left(-\frac {t^2}{\ln \frac  {x_1} {y_1}}\right)
\exp^\ast(-|y^\prime |^2) \lf(\lf[\ln\frac{x_1}{y_1}\r]^{\frac{\az} 2}
+|y'|^\alpha\r)\,dy'\,dy_1.
$$
Integrating  in $y'$ and noticing that $\ln\,(x_1/{y_1})\gs 1$,
we can control this double integral by
$$
\int_1^{x_1/2}
\frac t{y_1} \left(\ln \frac {x_1} {y_1}\right)^{ \frac{\az} 2-\frac32}  \,
\exp^\ast\left(-\frac {t^2}{\ln \frac  {x_1} {y_1}}\right)
\,dy_1.
$$
The transformation of variable $s = t^{-2}(\ln {x_1} - \ln {y_1})$ now
gives the desired bound $t^\az$.

Summing up, we have verified \eqref{GLip} and $\rm (i)$.
The norm equivalence \eqref{norm} also follows, and
this ends the proof of
Theorem \ref{thm1.2}.


\section{An example of a function in $\mathrm{GGLip}_\az$}\label{sect5}

With $\az\in(0,1)$,
we consider the function

$$f(x)=[\ln(e+|x|)]^{\az/2},\qquad x\in\rn.$$
We shall verify that $f$ belongs to $\mathrm{GGLip}_\az$,
using Theorem \ref{thm1.2}.

The estimate
\begin{equation}\label{eq5.1}
|f(x)-f(y)| \ls |x-y|^\az
\end{equation}
is easy and left to the reader.

To show that
\begin{equation}\label{eq5.2}
|f(x)-f(y)|\ls \lf|\ln(e+|x|)- \mathrm{sgn} \langle x,y\rangle  \ln(e+|y_x|)\r|^{\frac{\az} 2}
+|y_x'|^\alpha,
\end{equation}
 write
$$
|f(x)-f(y)|
\le |f(x)-f(y_x)|+|f(y_x)-f(y)|.
$$
The last term here is controlled by $|y_x'|^\alpha$, because of \eqref{eq5.1}.
To the first term on the right, we apply the inequality  $|a^\kz-b^\kz|\le |a-b|^\kz,  \; \; a,b>0$, with  $\kz = \az/2 \in(0,1)$, getting
$$
|f(x)-f(y_x)|
= \lf|[\ln(e+|x|)]^{\az/2}-[\ln(e+|y_x|)]^{\az/2} \r|
\le \lf|\ln(e+|x|)-\ln(e+|y_x|)\r|^{\frac{\az} 2}.
$$
This implies \eqref{eq5.2}, and it follows that
$f\in \mathrm{GGLip}_\az$.


\medskip

\noindent {\sc Liguang Liu\,}

\noindent  Department of Mathematics,
School of Information\\
Renmin University of China\\
Beijing 100872\\
China

\noindent {\it E-mail}: \texttt{liuliguang@ruc.edu.cn}

\bigskip

\noindent {\sc Peter Sj\"ogren\,}

\noindent  Mathematical Sciences,
University of Gothenburg\\ and\\
Mathematical Sciences, Chalmers\\
SE-412 96  G\"oteborg\\
Sweden

\noindent {\it E-mail}: \texttt{peters@chalmers.se}

\end{document}